\documentclass[a4paper,11pt,reqno]{amsart}
\usepackage{amsfonts}
\usepackage{amssymb}
\usepackage{amscd}
\usepackage{palatino}
\usepackage{amsthm}
\usepackage{a4wide}
\usepackage{mathrsfs}
\usepackage{tikz}
\usepackage{amsmath}
\usepackage{enumitem}

\usetikzlibrary{matrix,arrows,calc}

\usepackage{color}
\definecolor{citegreen}{rgb}{0,0.6,0}
\definecolor{refred}{rgb}{0.8,0,0}
\usepackage[colorlinks, urlcolor=blue, citecolor=citegreen,linkcolor=refred]{hyperref}

\usepackage[fixamsbugs]{mathtools}
\mathtoolsset{showonlyrefs,showmanualtags}

\theoremstyle{plain}

\newtheorem{teorema}{Theorem}[section]
\newtheorem{lemma}[teorema]{Lemma}
\newtheorem{proposizione}[teorema]{Proposition}
\newtheorem{corollario}[teorema]{Corollary}
\newtheorem{ackn}{Acknowledgments\!}

\theoremstyle{definition}
\newtheorem{definizione}[teorema]{Definition}

\theoremstyle{remark}

\newtheorem{osservazione}[teorema]{Remark}

\numberwithin{equation}{section}

\def\SSSS{{{\mathbb S}}}

\def\R{{\RR}}
\def\RR{{\mathbb R}}

\def\HH{{{\mathbb H}}}

\def\HHH{{\mathrm H}}

\def\Ric{{\mathrm {Ric}}}
\def\Riem{{\mathrm {Riem}}}

\def\Rm{{\mathrm {Rm}}}

\def\tr{\operatornamewithlimits{tr}\nolimits}

\newcommand{\Sec}{\mathrm{Sec}}

\newcommand{\id}{\mathrm{Id}}

\DeclareMathOperator{\vol}{Vol}

\DeclareRobustCommand{\rchi}{{\mathpalette\irchi\relax}}
\newcommand{\irchi}[2]{\raisebox{+1.5pt}{$#1\chi$}} 

\makeatletter
\patchcmd{\@setaddresses}{\indent}{\noindent}{}{}
\patchcmd{\@setaddresses}{\indent}{\noindent}{}{}
\patchcmd{\@setaddresses}{\indent}{\noindent}{}{}
\patchcmd{\@setaddresses}{\indent}{\noindent}{}{}
\makeatother

\title[Distance Functions, Curvature and Topology]{Distance Functions, Curvature and Topology}

\author[Carlo Mantegazza]{Carlo Mantegazza}
\address[Carlo Mantegazza]{Dipartimento di Matematica e Applicazioni ``Renato Caccioppoli'', Universit\`a di Napoli Federico II, Italy}
\email[C. Mantegazza]{carlo.mantegazza@unina.it}

\author[Francesca Oronzio]{Francesca Oronzio}
\address[Francesca Oronzio]{Scuola Superiore Meridionale, Napoli, Italy.}
\email[F. Oronzio]{f.oronzio@ssmeridionale.it}

\date{\today}



\usepackage{palatino}
\usepackage{a4wide}
\usepackage{anyfontsize}

\begin{document}

\hbadness=50000
\vbadness=50000

\begin{abstract} 
We discuss some properties of the distance functions on Riemannian manifolds and we relate their behavior to the geometry of the manifolds. This leads to alternative proofs of some ``classical'' theorems connecting curvature and topology.
\end{abstract}

\maketitle

\setcounter{tocdepth}{1}
\tableofcontents

\section{Introduction}

In this note we collect and organize in detail, discuss and revisit some material about the properties of the distance functions on Riemannian manifolds, scattered throughout the three editions of the book of Peter Petersen~\cite{petersen1,petersen2,petersen3}, in order to connect their behavior with the geometry of the manifolds and possibly have a point of view, different from the ``standard'' one based on the study of the geodesics, leading to alternative proofs of some ``classical'' theorems connecting curvature and topology. Heuristically (having in mind the context of the {\em calculus of variations}), focusing on the geodesic flow is considering the ``Lagrangian'' point of view in studying the length (or energy) functional, since geodesics are its critical points (solutions of its Euler--Lagrange equations, which are a system of ODEs), while considering the distance functions between points or subsets of the manifold  is taking the ``Eulerian'' point of view, as the distance functions are the ``value functions'' associated with the length functional (using a terminology from {\em optimal control theory}, see~\cite{barcap} for instance). The straightforward connection between the two approaches is clearly given by the fact that the distance functions are realized by the length of the minimal geodesics and conversely, the geodesics are locally the integral curves of the gradients of the distance functions.

This connection also offers the possibility of using more ``analytic'' tools, related to the general study of the distance functions, seeing them simply as solutions of the {\em Hamilton--Jacobi} {\em (eikonal) equation} $\Vert \nabla u\Vert=1$ (see~\cite{mant4}). Following Petersen, we describe this approach in general in the next section, then we specialize it to the particular distance functions from a fixed point (in Section~\ref{The distance function from a point}), in order to connect the curvature and the topology/geometry of the manifold. The formulas we obtain give us information that can substitute for what can be inferred about the behavior of geodesics, in the ``classical'' approach, using the second variation of the length functional and the theory of Jacobi fields.

As examples of application, we will show in Section~\ref{curvtop} how this information can be used to obtain alternative proofs (possibly easier or more intuitive than the ``standard'' ones) of the existence of local isometries between Riemannian manifolds of equal constant sectional curvature, of the Bonnet--Myers and Cartan--Hadamard theorems, of Synge's lemma, as well as estimates for the volume growth of geodesic balls. 

Finally, it is also worth underlining that the approach to these results, which belong to the field called {\em comparison geometry}, is to describe locally the manifolds in polar coordinates and to ``compare'' the induced metric and the curvature of the geodesic spheres (not geometrically, but analytically, by studying the ODEs satisfied by their components along geodesics) with the analogous ones of appropriate space forms.

\medskip

{\em Throughout the paper, we will consider only connected, complete, $n$--dimensional Riemannian manifolds with $n\geqslant 2$ and $\nabla$ will be the Levi--Civita connection. Moreover, we will always use the Einstein convention of summing over repeated indices.}

\begin{ackn}
All authors are members of INdAM--GNAMPA. The first author is partially supported by the PRIN Project 2022E9CF89 ``GEPSO -- Geometric Evolution Problems and Shape Optimization''.
\end{ackn}

\section{Distance functions}\label{eriemsec}

References for the material of this section are~\cite[Sections~4--5, Chapter~2]{petersen2} and~\cite[Section~3.2]{petersen3} (see also~\cite[Sections~2.3, 2.4]{petersen1} and~\cite{mant4}).

\begin{definizione}
Let $(M,g)$ be a Riemannian manifold and let $r:U\rightarrow \R $ be a smooth function, where $U$ is an open set of $M$. We say that $r$ is a {\em distance function} if it satisfies the {\em Hamilton--Jacobi equation} $\Vert \nabla r\Vert=1$ ({\em eikonal equation}) on $U$.
\end{definizione}

A classic example is the distance function $d_p:M\to\R$ from a point $p\in M$, which is a distance function outside of $\{p\}$ and the cut~locus of $p$. The same holds for $r=d_K$ if $K$ is a smooth embedded submanifold of $M$, without boundary and closed as a subset of $M$ (see~\cite{mant4} for the notion of cut~locus of a submanifold and the analysis of $d_K$ when $K$ is only of class $C^k$ or merely a closed set). Moreover, if $K$ is a hypersurface, in a small ball $B$ of $M$ centered at a point $p\in K$, we can consider the {\em signed distance function} $\widetilde{d}_K$ from $K$, which coincides with $d_K$ in one of the two connected components of $B\setminus K$ and with $-d_K$ in the other. It is then easy to see that this function is smooth in the whole ball $B$, also at the points of $K$. We remark that restricting ourselves to considering only these functions would not be particularly limiting, since at least locally every distance function $r:U\to\R$ coincides with the function $\widetilde{d}_{\Sigma_\rho}+\rho$, where $\Sigma_\rho=\{r=\rho\}$, that is, $r$ is locally equal to a signed distance function, up to a constant.

\smallskip

Given a distance function $r:U\rightarrow \R $ in $(M,g)$, we will denote by $\partial_r$ the gradient vector field $\nabla r$, by $\Sigma_\rho=\left\{\,p\in U:r(p)=\rho\,\right\}$ the level sets of $r$, which (if not empty) are smooth (embedded) hypersurfaces of $(M,g)$, by the implicit function theorem, since $\Vert \nabla r\Vert =1$ on $U$. Then, for every $p\in \Sigma_\rho$, the tangent space $T_{p}\Sigma_\rho$ is given by the kernel of $dr_p:T_pM\to\R$, which coincides with the space of vectors $v\in T_{p}M$ orthogonal to $\partial_r|_p$. We will denote by $g^{\rho}$, $\nabla^{\rho}$, $\Rm^{\rho}$ and $\Riem^\rho$, the induced metric on $\Sigma_\rho$, the Levi--Civita connection, the Riemann  operator and tensor of $(\Sigma_\rho,g^{\rho})$, respectively and if there is no risk of ambiguity, we will drop the superscript $\rho$.

\medskip

{\em We adopt the convention that 
$$
\Rm(X,Y)Z= \nabla_{X}\nabla_{Y}Z-\nabla_{Y}\nabla_{X}Z-\nabla_{\left[X,Y\right]}Z
$$
and
$$
\Riem(X,Y,Z,W)=g(\Rm(X,Y)Z,W),
$$
while the sectional curvature of a plane generated by two linearly independent vectors $v,w\in T_pM$ is given by
$$
\Sec(v,w)=\frac{\Riem(v,w,w,v)}{\Vert v\Vert_{p}^{2}\Vert w\Vert_{p}^{2}-g_{p}^{2}( v,w)}.
$$
}

\medskip

We define $\mathrm{S}:\Gamma(TU)\to\Gamma(TU)$ as
$$
\mathrm{S}(X)=\nabla_{X} \partial_r,
$$
for every vector field $X$ on $U$. It is the $(1,1)$--version of the Hessian of $r$. Indeed,
$$
g(\mathrm{S}(X),Y)= g(\nabla_{X} \partial_r,Y)=\mathrm{Hess}\,r(X,Y)=\mathrm{Hess}\,r(Y,X)=g(X,\nabla_{Y} \partial_r)=g(X,\mathrm{S}(Y)).
$$
By means of the tensor $\mathrm{S}$, we can define the $(1,1)$--tensor $\mathrm{S}^2$ as $\mathrm{S}\circ \mathrm{S}$ and the $(0,2)$--tensor $\mathrm{Hess}^2r$ by
$$
\mathrm{Hess}^2r(X,Y)=g(\mathrm{S}^2(X),Y)=g(\mathrm{S}(X),\mathrm{S}(Y))=\mathrm{Hess}\,r(\mathrm{S}(X),Y).
$$
The latter is clearly positive semidefinite. We also notice that in a local chart, there hold
$$
\tr(\mathrm{S})=g(\mathrm{S}(\partial_i),\partial_j)g^{ij}=g^{ij}\mathrm{Hess}_{ij}r=g^{ij}\nabla^2_{ij}r=g^{ij}\nabla_i\nabla_jr=\Delta r
$$
and
$$
\tr(\mathrm{S}^2)=g(\mathrm{S}^2(\partial_i),\partial_j)g^{ij}=g(\mathrm{S}(\partial_i),\mathrm{S}(\partial_j))g^{ij}=\mathrm{S}_i^k\mathrm{S}_j^\ell g^{ij}g_{k\ell}=\Vert \mathrm{S}\Vert^2=\Vert\mathrm{Hess}\,r\Vert^2,
$$
where $\partial_i$ are the coordinate vector fields.

The {\em Lie derivative} of a tensor will be denoted by $L$. In particular, if $X,Y$ are two vector fields, we have $L_XY=[X,Y]$, in local coordinates,
$$
L_XY=\Bigl(X^i\frac{\partial Y^j}{\partial x^i}-Y^i\frac{\partial X^j}{\partial x^i}\Bigr)\partial_j,
$$
obviously $L_XY=-L_YX$ and finally $L_XX=0$.\\
We then underline the relation
$$
\nabla_{\partial_r}X=\nabla_X\partial_r+[\partial_r,X]=\mathrm{S}(X)+L_{\partial_r}X,
$$
that is,
\begin{equation}\label{eq800}
\nabla_{\partial_r}X-L_{\partial_r}X=\mathrm{S}(X).
\end{equation}
It says that the difference between the covariant and the Lie derivative with respect to $\partial_r$ of a vector field $X$ on $U$ is given by the tensor $\mathrm{S}$ applied to $X$.

\begin{lemma}\label{partialrautovettorediS}
Let $r:U\rightarrow \R $ be a distance function. Then we have
$$ 
\nabla_{\partial_{r}} \partial_{r}=0.
$$
Hence, $\mathrm{S}(\partial_r)=0$, that is, $\partial_r$ is an eigenvector of $\mathrm{S}$ with eigenvalue zero and $\mathrm{Hess}\,r(\partial_r,\cdot)=0$.
\end{lemma}
\begin{proof}
For every vector field $X$ on $U$, we have
$$
g(\nabla_{\partial_{r}} \partial_{r},X)=\mathrm{Hess}\,r(\partial_{r},X)=\mathrm{Hess}\,r(X,\partial_{r})=g(\nabla_{X}\partial_{r},\partial_{r})=Xg(\partial_{r},\partial_{r})/{2}=0,
$$
where we used the symmetry of the Hessian and $g(\partial_{r},\partial_r)=1$.
\end{proof}

A direct consequence of Lemma~\ref{partialrautovettorediS} is that any integral curve for $\partial_r$ is a geodesic of $(M,g)$, parametrized by arc length.

\begin{osservazione} Obviously, by definition, we also have $L_{\partial_{r}} \partial_{r}=0$.
\end{osservazione}

The field $\partial_r$ is the unit normal vector field pointing in the direction of increasing $r$ along any nonempty $\Sigma_\rho$ and, for every $X\in \Gamma (T\Sigma_\rho)$, the vector field $\mathrm{S}(X)$ along $\Sigma_\rho$ belongs precisely to $\Gamma(T\Sigma_\rho)$. Then, the {\em shape operator} of $\Sigma_\rho$ coincides with the map that associates $\mathrm{S}(X)\in\Gamma (T\Sigma_\rho)$ to each $X\in\Gamma (T\Sigma_\rho)$, therefore it will be still denoted by $\mathrm{S}$.
The associated symmetric $(0,2)$--tensor on the hypersurface $\Sigma_\rho$, defined by
$$
\mathrm{h}(X,Y)=-g(\nabla_XY,\partial_r)=g(\mathrm{S}(X),Y)=\mathrm{Hess}\,r(X,Y)
$$
for every pair of tangent vector fields $X,Y$ to $\Sigma_\rho$, is called the {\em second fundamental form} $\mathrm{h}$. Its trace or the trace of the shape operator $\mathrm{S}$ with respect to $g^{\rho}$ is the {\em mean curvature} $\HHH$ of $\Sigma_\rho$. Notice that $\HHH=\tr \mathrm{S}$ and $\Vert\mathrm{h}\Vert^2=\tr \mathrm{S}^2$, by Lemma~\ref{partialrautovettorediS}.

\begin{proposizione}[Curvature equations]\label{equazionidellacurvatura}
The following equalities hold~\cite[Theorems~2-4, Section~4.2, Chapter~2]{petersen2},
\begin{enumerate}
\item $\nabla_{\partial_{r}}\mathrm{S}+\mathrm{S}^{2}=\Rm(\partial_{r},\cdot)\partial_{r}$
\item $\Riem(X,Y,Z,W)=\Riem^{\rho}(X,Y,Z,W)+{\mathrm{h}}(X,Z){\mathrm{h}}(Y,W)-\mathrm{h}(Y,Z){\mathrm{h}}(X,W)$
\item $\Riem(X,Y,Z,\partial_{r})=\nabla^\rho_{Y}{\mathrm{h}}(X,Z)-\nabla^\rho_{X}{\mathrm{h}}(Y,Z)$
\end{enumerate}
for any $X,Y,Z,W$ tangent vector fields to $\Sigma_\rho$.
\end{proposizione}

The first identity is called {\em radial curvature equation} (see also~\cite[Theorem~2.4 -- Equation~2.5]{ambman1} and~\cite[Theorem~3.2]{ambson} where the ambient is the Euclidean space and the codimension is arbitrary), while the second and the third one are the classical {\em Gauss} and {\em Codazzi--Mainardi equations}.

\begin{proposizione}[Equations of Riemannian geometry]\label{efgr} 
The following equalities hold~\cite[Proposition~7, Section~5.1, Chapter~2]{petersen2},
\begin{enumerate}
\item $\nabla_{\partial_{r}}g=0$
\item $L_{\partial_{r}}g=2 \mathrm{Hess}\,r$
\item $L_{\partial_{r}}\mathrm{S}=\nabla_{\partial_{r}}\mathrm{S}=-\mathrm{S}^{2}+\Rm(\partial_{r},\cdot)\partial_{r}$
\item $\nabla_{\partial_{r}} \mathrm{Hess}\,r(X,Y) + \mathrm{Hess}^{2}r(X,Y)=-\Riem(X,\partial_{r},\partial_{r},Y)$
\item $L_{\partial_{r}} \mathrm{Hess}\,r (X,Y) - \mathrm{Hess}^{2}r(X,Y)=-\Riem(X,\partial_{r},\partial_{r},Y)$
\end{enumerate}
for any $X,Y$ vector fields on $U$.
\end{proposizione}

The second equation shows how the Hessian of the distance function $r$ ``controls'' the metric $g$, while the fourth and fifth equations describe how the curvature influences this Hessian. Now,
 the vector fields satisfying $\nabla_{\partial_r}X=0$, as well as those satisfying $L_{\partial_r}X=0$, are particularly useful in order to exploit such information. It is a consequence of the following corollary.

\begin{corollario}\label{efgrcor} 
We have the following equalities
\begin{align*}
\partial_r g(X,Y)=&\,g(\nabla_{\partial_r}X,Y)+g(X,\nabla_{\partial_r}Y)\\
=&\,2 \mathrm{Hess}\,r(X,Y)+g(L_{\partial_r}X,Y)+g(X,L_{\partial_r}Y)\\
\partial_r\mathrm{Hess}\,r(X,Y)=&\, - \mathrm{Hess}^{2}r(X,Y)-\Riem(X,\partial_{r},\partial_{r},Y)\\
&\,+\mathrm{Hess}\,r(\nabla_{\partial_r}X,Y)+\mathrm{Hess}\,r(X,\nabla_{\partial_r}Y)\\
=&\,\mathrm{Hess}^{2}r(X,Y)-\Riem(X,\partial_{r},\partial_{r},Y)\\
&\,+\mathrm{Hess}\,r(L_{\partial_r}X,Y)+\mathrm{Hess}\,r(X,L_{\partial_r}Y)
\end{align*}
for every pair $X,Y$ of vector fields on $U$.
\end{corollario}

\begin{definizione} Given a distance function $r:U\rightarrow \R$, we say that the vector field $J$ on $U$ is a {\em Jacobi vector field for $r$} if $L_{\partial_{r}} J=0$ and we say that $E$ is a {\em parallel vector field for $r$} if $\nabla_{\partial_{r}} E=0$ in $U$.
\end {definizione}

For every Jacobi field $J$ for $r$, there holds
\begin{equation}\label{fo:3}
\nabla_{\partial_r}J=\mathrm{S}(J)=\nabla_J\partial_r
\end{equation}
in $U$, by relation~\eqref{eq800}.\\
Assuming that the two vector fields in Corollary~\ref{efgrcor} are Jacobi fields $J_1,J_2$ for $r$, we get (in the ``second lines'' of the equalities),
\begin{align}
\partial_r g(J_1,J_2)=&\,2 \mathrm{Hess}\,r(J_1,J_2)\label{JJJ1}\\
\partial_r\mathrm{Hess}\,r(J_1,J_2)=&\,\mathrm{Hess}^{2}r(J_1,J_2)-\Riem(J_1,\partial_{r},\partial_{r},J_2).
\end{align}
In particular, since $\partial_r$ is also a Jacobi field for $r$, 
\begin{equation}
\partial_r g(\partial_r,J)=2 \mathrm{Hess}\,r(\partial_r,J)=0\qquad\text{ and }\qquad\partial_r g(J,J)=2 \mathrm{Hess}\,r(J,J),\label{eq810}
\end{equation}
for every Jacobi field $J$ for $r$. This clearly implies that if $J$ is orthogonal to $\partial_r$ at some point on an integral curve for $\partial _r$, it is orthogonal to $\partial_r$ on the whole curve. Then, if $J$ is not pointwise proportional to $\partial_{r}$, there  holds
\begin{align}
\partial_r\mathrm{Hess}\,r(J,J)=&\,
\mathrm{Hess}^{2}r(J,J)-\Sec(J,\partial_{r}) g(J-g(J,\partial_{r})\partial_{r},J-g(J,\partial_{r})\partial_{r})\\
\geqslant&\,[\mathrm{Hess}\,r(J,J)]^2/g(J,J)-\Sec(J,\partial_{r}) g(J-g(J,\partial_{r})\partial_{r},J-g(J,\partial_{r})\partial_{r}),
\end{align}
where we used the elementary inequality
\begin{equation}
[\mathrm{Hess}\,r(X,X)]^2=[g(\mathrm{S}(X),X)]^2\leqslant g(\mathrm{S}(X),\mathrm{S}(X))g(X,X)=\mathrm{Hess}^{2}r(X,X)g(X,X),\label{disS}
\end{equation}
which holds for every vector field $X$ on $U$.\\
Thus, if $J$ is a Jacobi vector field for $r$ pointwise orthogonal
to $\partial_r$, then along any integral curve of $\partial_r$
on which $J$ does not vanish we have
\begin{align}
\partial_r\mathrm{Hess}\,r(J,J)\geqslant&\,[\mathrm{Hess}\,r(J,J)]^2/g(J,J)-\Sec(J,\partial_{r}) g(J,J)\label{eq:fo2}\\
\partial_{rr}^2 g(J,J)/2\geqslant&\,[\mathrm{Hess}\,r(J,J)]^2/g(J,J)-\Sec(J,\partial_{r}) g(J,J)\geqslant-\Sec(J,\partial_{r}) g(J,J).
\end{align}

Notice that $\partial_r=\nabla r$ is a parallel (and Jacobi) vector field for $r$. Moreover, for every parallel vector field $E$ for $r$, there holds
$$
L_{\partial_r}E=-\mathrm{S}(E)=-\nabla_E\partial_r
$$
in $U$, again by the relation~\eqref{eq800}.\\
Assuming that the two vector fields in Corollary~\ref{efgrcor} are parallel fields $E_1,E_2$ for $r$, we get (in the ``first lines'' of the equalities),
\begin{align*}
\partial_r g(E_1,E_2)=&\,0\\
\partial_r\mathrm{Hess}\,r(E_1,E_2)=&\,-\mathrm{Hess}^{2}r(E_1,E_2)-\Riem(E_1,\partial_{r},\partial_{r},E_2).
\end{align*}
In particular, since $\partial_r$ is also a parallel field for $r$, 
$$
\partial_r g(\partial_r,E)=\partial_r g(E,E)=0
$$
for every parallel field $E$ for $r$. This clearly implies that if $E$ is orthogonal to $\partial_r$ at some point on an integral curve for $\partial _r$, it is orthogonal to $\partial_r$ on the whole curve and its norm is constant. Then, if $E$ is not pointwise proportional to $\partial_{r}$, there holds
$$
\begin{aligned}
\partial_r\mathrm{Hess}\,r(E,E)=&
-\mathrm{Hess}^{2}r(E,E)-\Sec(E,\partial_{r}) g(E-g(E,\partial_{r})\partial_{r},E-g(E,\partial_{r})\partial_{r})\\
\leqslant&-[\mathrm{Hess}\,r(E,E)]^2/g(E,E)-\Sec(E,\partial_{r}) g(E-\!g(E,\partial_{r})\partial_{r},E-\!g(E,\partial_{r})\partial_{r}),\label{eq813}
\end{aligned}
$$
using again inequality~\eqref{disS}.\\
Hence, if $E$ is a parallel vector field for $r$ which is pointwise orthogonal to $\partial_r$ and $g(E,E)=1$, we have
\begin{equation}
\partial_r\mathrm{Hess}\,r(E,E)\leqslant-[\mathrm{Hess}\,r(E,E)]^2-\Sec(E,\partial_{r}).\label{eq510}
\end{equation}

Assuming a sign condition on the sectional curvature of $(M,g)$, by means of Jacobi/parallel fields for $r$, these formulas give differential information on $g$. More precisely, using parallel fields and the last formula above, we obtain information about the evolution of the Hessian along an integral curve of $\partial_r$ from information on the sign of the sectional curvature. Afterwards, by using Jacobi fields, we ``transfer'' such information on the Hessian to the metric by means of equation~\eqref{JJJ1}. 

We can always choose local coordinate systems $(x^1,\dots,x^{n-1},r)$, given by the flow of the vector field $\nabla r$ (this can be done since $\Vert\nabla r\Vert=1$ in $U$). We call these coordinates {\em adapted to $r$}. In a coordinate system adapted to $r$, letting $J=a^\alpha\partial_\alpha+a^{r}\partial_{r}$, it follows that $J$ is a Jacobi field for $r$ if and only if the functions $a^\alpha, a^{r}$ are independent of $r$ (obviously this implies that the coordinate fields are Jacobi fields for $r$). More generally, in such coordinates the first--order linear PDEs of the Jacobi and parallel vector fields can be treated by considering $r$ as the independent variable and solving the associated ODEs. Thus, if we set (or we know) the value of a vector field on a hypersurface $\Sigma_\rho$, we can extend it (uniquely) at least locally, to a Jacobi or parallel vector field for $r$. This is known as the {\em method of characteristics} (see~\cite{evans2}, for instance). Therefore, if $J$ is a Jacobi field for $r$ not proportional to $\partial_r$ at a point of an integral curve of $\partial_r$, this holds for the whole curve. The same conclusion is true also for parallel fields for $r$.

In coordinates adapted to $r$, the metric is then given by
$$
g = g_{\alpha\beta} \, dx^\alpha\! \otimes dx^\beta+dr\otimes dr,
$$
hence the symmetric $n\times n$ matrix of its components has the form
$$ 
g=\left(\begin{matrix} g_{11}&\dots&g_{1n-1} &0\\ \dots& \dots & \dots& \dots \\ g_{n-11}&\dots&g_{n-1 n-1} &0\\ 0&\dots&0&1 \end{matrix}\right).
$$
Denoting by $\mathrm{h}_{ij}=\mathrm{Hess}\,r(\partial_i,\partial_j)$ the components of the Hessian of $r$ and recalling that $\mathrm{Hess}\,r(\partial_r,\cdot)=0$, we have $\mathrm{h}_{ij}=\mathrm{h}_{ji}$ and $\mathrm{h}_{rr}=\mathrm{h}_{\alpha r}=\mathrm{h}_{r\alpha}=0$, as a consequence,
\begin{equation}\label{eq700}
\mathrm{S}^{\beta}_{\alpha}=\mathrm{h}_{\alpha \gamma}g^{\gamma \beta} \quad \quad\text{ and }\quad\quad \mathrm{S}_{\alpha}^{r}= \mathrm{S}_{r}^{\alpha}= \mathrm{S}_{r}^{r}=0 .
\end{equation}
Hence, the $n\times n$ matrix of the components of $\mathrm S$ is
$$
\mathrm{S}=\left(\begin{matrix} \mathrm{S}_{1}^{1}& \dots & \mathrm{S}_{n-1}^{1}&  0\\ \dots& \dots & \dots& \dots  \\ \mathrm{S}_{1}^{n-1} & \dots & \mathrm{S}_{n-1}^{n-1}&  0\\  0 & \dots & 0 & 0\end{matrix}\right).
$$
Finally, we define the symmetric $(0,2)$--tensor 
$$
\mathrm R(X,Y)=\Riem(X,\partial_r,\partial_r,Y),
$$
with associated self--adjoint operator $R$ defined by $\mathrm R(X,Y)=g(R(X),Y)$, with component matrix
$$
R=\left(\begin{matrix} R_{1}^{1}& \dots & R_{n-1}^{1}&  0\\ \dots& \dots & \dots& \dots  \\ R_{1}^{n-1} & \dots & R_{n-1}^{n-1}&  0\\  0 & \dots & 0 & 0\end{matrix}\right)
$$
where
\begin{equation}\label{eq701}
R^{\beta}_{\alpha}=\mathrm R_{\alpha \gamma}g^{\gamma \beta} \quad \quad\text{ and }\quad\quad R_{\alpha}^{r}= R_{r}^{\alpha}= R_{r}^{r}=0,
\end{equation}
noticing that $\tr R=\Ric(\partial_r,\partial_r)$.

We can then rewrite some of the previous equations in these coordinates.

\begin{proposizione}\label{eqriempol}
The following equations hold
\begin{enumerate}\label{efgrincoordinatepolari}
\item $\partial_{r}g_{\alpha \beta}=2\mathrm h_{\alpha \beta}=2 \mathrm S_\alpha^\gamma g_{\gamma \beta}$
\item $\partial_{r} g^{\alpha \beta}=-2\mathrm h^{\alpha \beta}=-2 g^{\alpha \gamma} \mathrm h_{\gamma\lambda} g^{\lambda\beta}=-2 \mathrm S_\gamma^\beta g^{\gamma\alpha}$
\item $\partial_{r}\mathrm h_{\alpha \beta} = \mathrm h_{\alpha \gamma} g^{\gamma\lambda} \mathrm h_{\lambda\beta}-\Riem_{\alpha rr\beta }$
\item $\partial_{r}\mathrm S_\alpha^\beta = -\mathrm S_\alpha^\gamma \mathrm S_\gamma^\beta-R_\alpha^\beta$.
\end{enumerate}
\end {proposizione}
 
Letting $G=\sqrt{\det g_{ij}\,}=\sqrt{\det g_{\alpha\beta}\,}$ (the density of volume) and recalling that $\HHH=\tr \mathrm S$, the following proposition is then a straightforward consequence (using the arithmetic--quadratic mean inequality).

\begin{proposizione}\label{eqS}
There hold
\begin{enumerate}
\item $\partial_{r}\log G=\tr \mathrm S=\HHH$
\item $\partial_{r}\HHH=-\tr \mathrm S^{2}-\Ric(\partial_{r},\partial_{r})=-\Vert\mathrm{h}\Vert^2-\Ric(\partial_{r},\partial_{r})\leqslant\displaystyle{-\frac{\HHH^2}{n-1}-\Ric(\partial_{r},\partial_{r})}$.
\end{enumerate}
\end{proposizione}

Examining the system of ODEs for the matrices $g_{\alpha \beta}$ and $\mathrm S_\alpha^\beta$, seen as functions of $r$ (by evaluating them along an integral curve of $\partial_r$ which is a geodesic parametrized by arc length),
$$
\begin{cases}
\,\partial_r g_{\alpha \beta }=2 {\mathrm S}_\alpha^\gamma g_{\gamma \beta }\\
\,\partial_r {\mathrm S}_\alpha ^\beta =-{\mathrm S}_\alpha^\gamma {\mathrm S}_\gamma^\beta-R_\alpha^\beta
\end{cases}
$$
given by the first and the last equations in Proposition~\ref{eqriempol}, we notice that it is partially decoupled, that is, $g_{\alpha\beta}$ is not present in the second equation, which is of {\em Riccati type} (quadratic), while the first equation is linear.\\
Then, if $\mu_{\min}=\mu_{\min}(r)$ and $\mu_{\max}=\mu_{\max}(r)$ are the minimum and maximum eigenvalues of the matrix $\mathrm S_\alpha^\beta(r)$, with corresponding $g$--unit eigenvectors $v=v(r)=v^{\alpha}\partial_{\alpha}$ and $w=w(r)=w^{\beta}\partial_{\beta}$ respectively, these functions are locally Lipschitz and satisfy the ODEs
\begin{align}
\mu_{\min}'=&\, -\mu_{\min}^2-R_\alpha^\beta v^\alpha v^\gamma g_{\gamma \beta}=\, -\mu_{\min}^2-\mathrm{R}_{\alpha\gamma} v^\alpha v^\gamma \geqslant-\mu_{\min}^2-{\mathrm{MaxSec}} \label{eq333-0}\\
\mu_{\max}'=&\, -\mu_{\max}^2-R_\alpha^\beta w^\alpha w^\gamma g_{\gamma \beta}=\,-\mu_{\max}^2-\mathrm{R}_{\alpha \gamma} w^\alpha w^\gamma \leqslant -\mu_{\max}^2-{\mathrm{MinSec}}, \label{eq333}
\end{align}
for almost every $r$, where ${\mathrm{MinSec}}$ and ${\mathrm{MaxSec}}$ denote the pointwise minimum and maximum sectional curvatures along an integral curve $r\mapsto\gamma(r)$ of $\partial_r$. Indeed, the matrix $\mathrm S_\alpha^\beta$ is similar to the matrix $\widetilde{\mathrm S}_\alpha^\beta$ obtained by representing $\mathrm S$ with respect to some smooth orthonormal frame $\{E_1,\,\dots, E_{n-1}, \partial_r\}$. The matrix $\widetilde{\mathrm S}_\alpha^\beta$ is symmetric and has smooth entries, hence, by~\cite[Section~6, Chapter~2]{Kato} there exist $n-1$ functions $\mu_i(r)$, at least of class $C^1$, that represent the (possibly repeated) eigenvalues of $\widetilde{\mathrm S}_\alpha^\beta(r)$, the same as those of $\mathrm S_\alpha^\beta(r)$, therefore, considering the function $\mu(i,r)=\mu_i(r)$ and following the proof of~\cite[Theorem~2.1.1]{Manlib}, we can conclude the local Lipschitz continuity of $\mu_{\max}(r)=\max_{i=1,\dots,n-1}\mu(i,r)$. Similarly we can prove the local Lipschitz continuity of $\mu_{\min}(r)$. Moreover, for every value $r_0$ at which the function $\mu_{\min}(r)$ is differentiable, we can consider the auxiliary function $$\varphi(r)=\mathrm S_\alpha^\beta(r)v^\alpha(r_0) v^\gamma(r_0)g_{\gamma \beta}(r)/(v^\iota(r_0) v^\kappa(r_0)g_{\iota \kappa}(r))-\mu_{\min}(r),$$ which has $r_0$ as a point of minimum. Consequently,
$$
\mu_{\min}'(r_0)=\, -\mu_{\min}^2(r_0)-\mathrm{R}_{\alpha\gamma}(r_0) v^\alpha(r_0) v^\gamma(r_0)\geqslant \, -\mu_{\min}^2(r_0)-{\mathrm{MaxSec}}.
$$
Similarly, we can obtain the differential inequality~\eqref{eq333} and analogously, there hold
\begin{align}
\lambda'_{\min}=&\,2 \lambda_{\min}{\mathrm S}^\beta_\alpha y^\alpha y^\beta\geqslant 2\lambda_{\min}\mu_{\min} \label{eq333-0_bis}\\
\lambda'_{\max}=&\,2 \lambda_{\max}{\mathrm S}^\beta_\alpha z^\alpha z^\beta \leqslant 2\lambda_{\max}\mu_{\max} \label{eq333-bis},
\end{align}
for almost every $r$, where $\lambda_{\min}=\lambda_{\min}(r)$ and $\lambda_{\max}=\lambda_{\max}(r)$ are the minimum and maximum eigenvalues of the matrix $g_{\alpha\beta}$ with corresponding Euclidean unit eigenvectors $y=y(r)$ and $z=z(r)$, respectively.

Thus, if the curvature of $(M,g)$ is bounded below, the inequality~\eqref{eq333}, which gives $\mu_{\max}'\leqslant -{\mathrm{MinSec}}$, implies that $\mu_{\max}$ is bounded above on each interval $[a,b)$, with $a$ fixed in the domain and $b<+\infty$ (by integrating this last differential inequality and noticing that this can be done by the absolute continuity of $\mu_{\max}$). Then, $\mu_{\max}$ could possibly go to $-\infty$ (hence, {\em a fortiori} also $\mu_{\min}$), but it cannot diverge to $+\infty$.
Therefore, by inequality~\eqref{eq333-bis}, the positive function $\lambda_{\max}$, hence $\lambda_{\min}$, is also bounded above on the interval $[a,b)$. Accordingly, the only possible ``degeneration''  is that $\lambda_{\min}$ goes to zero as $r\to r_0^-$, for some $r_0$, meaning that the matrix $g_{\alpha\beta}$ becomes degenerate when $r$ approaches the value $r_0$.\\
This is clearly equivalent to the fact that the volume density $G$ goes to zero, as $r\to r_0^-$, indeed, an analogous analysis can be performed considering the equations for the pair $G$ and $\HHH$, in Proposition~\ref{eqS}.  

If $(M,g)$ is complete, $r$ is the distance function from a point $p\in M$ and $\lambda_{\min} \to 0$ as $r\to r_0^{-}$ along a geodesic $\sigma:[0,+\infty)\to M$ outgoing from $p$, parametrized by arc length and minimizing along the interval $[0,r_0)$, then $\sigma$ no longer minimizes the distance beyond $r_0$. Fix a polar chart whose domain contains $\sigma(0,r_0)$. There exists a map $r\mapsto(v^{1},\dots,v^{n-1})(r)\in \R^{n-1}$ such that $|v(r)|_{\R^{n-1}}=1$ for all $r\in(0,r_{0})$ and
$$
\left(\begin{matrix} g_{11}&\dots& g_{1n-1} \\  \dots& \dots & \dots \\ g_{n-1  1}&\dots&g_{n-1  n-1} \end{matrix}\right)(r)\left(\begin{matrix}v^{1}\\ \cdots \\ v^{n-1}\end{matrix}\right)(r) \to 0,\ \ \  \text {as}  \ \ \ r\to r_0^{-}.
$$
Hence,
$$
g_{\sigma(r)}\bigl(v^{\alpha}(r)\partial_{\theta^{\alpha}}|_{\sigma(r)},v^{\beta}(r)\partial_{\theta^{\beta}}|_{\sigma(r)}\bigr)=v^{\alpha}(r)v^{\beta}(r)g_{\alpha \beta}(r) \to 0,
$$
as $r\to r_0^{-}$. Then, there is a sequence $r_k\to r_0^-$ and there exists a vector $(\overline v^{1},\dots,\overline v^{n-1})\in \R^{n-1}$ such that $|\overline v|_{\R^{n-1}}=1$ and $(v^{1},\dots,v^{n-1})(r_k)\to (\overline v^{1},\dots,\overline v^{n-1})$, as $k\to\infty$. Then, since there holds
$$
\partial_{\theta^{\alpha}}|_{\sigma(r)}=d(\exp_{p})_{r\sigma'(0)}(r \eta_{\alpha}),
$$  
on $(0,r_{0})$, where $\eta_{1},\dots,\eta_{n-1}$ are linearly independent vectors in $T_pM$ which form a basis of $T_{\sigma'(0)}\SSSS^{n-1}_{p}$, by the continuity of $d\exp_p$ and of the metric $g$, it follows that 
$$
d(\exp_{p})_{r_{0}\sigma'(0)}(\overline v^{\alpha}\eta_{\alpha})=0.
$$
The components $\overline v^{\alpha}$ are not all zero and $\eta_{1},\dots,\eta_{n-1}$ are linearly independent, therefore, $d\exp_p$ is singular at $r_{0}\sigma'(0)$. Hence, $\sigma(r_0)$ is a {\em conjugate point} to $p$ along $\sigma$. Accordingly, the geodesic $\sigma$ cannot be minimal after $\sigma(r_0)$. See~\cite[Section~4]{mant4}, for more information about the set of focal points of a submanifold.

\smallskip

In the next section we will introduce the last ``analytical'' tools which, together with the above differential equations and differential inequalities, we will apply to the study of the relationships between curvature and topology. More precisely, we will analyze the behavior of the distance function $r$ from a point, according to the following lines:

\begin{enumerate}
\item If the sectional curvature $\Sec$ of the Riemannian manifold is bounded from above by a constant $K$, then the minimum of the eigenvalues of the shape operator $\mathrm S$ of $r$ is greater than or equal to a quantity depending only on the constant $K$ and on the value of the function $r$ at the point considered (by using the Jacobi fields for $r$). In the case that $K$ is nonpositive, this quantity is always positive (this is used in the proof of the key Proposition~\ref{CH} for the Cartan--Hadamard theorem).
\item If the sectional curvature $\Sec$ of the Riemannian manifold is bounded from below by a constant $k$, then the maximum of the eigenvalues of the shape operator $\mathrm S$ of $r$ is lower than or equal to a quantity depending only on the constant $k$ and on the value of $r$ at the point considered (by using parallel fields for $r$). If $k$ is positive, this quantity tends to $-\infty$ as $r\to r_{0}^-$, therefore the matrix of the coefficients of the metric becomes degenerate and thus all the geodesics parametrized by arc length cease to be minimizing in the same interval of length $r_{0}$. Consequently, the diameter of the manifold is finite (Synge's lemma~\ref{Synge}).
\item If the sectional curvature is constant, then the eigenvalues of the shape operator are all equal to the same quantity depending only on the constant of the sectional curvature and on the value of $r$ at the point considered. Therefore, in two Riemannian manifolds with the same constant sectional, in geodesic balls, diffeomorphic via polar coordinates, the coefficients of the metrics of both manifolds satisfy the same system of first order linear differential equations along the integral curves of the gradient of respective distance functions, with the same limit conditions (this will be used in Proposition~\ref{constsec}).
\item If the eigenvalues of the Ricci tensor $\Ric$ are bounded from below by a positive constant $l$, in dimension $n\geq3$ we can not obtain information on the individual eigenvalues of the shape operator $\mathrm S$ of $r$  but only on their sum, namely on $\tr(\mathrm S)$, which is bounded from above by a quantity depending on the constant $l$ and on the value of $r$ at the point considered. This quantity tends to $-\infty$ for $r\to r_{0}^-$ and therefore the diameter of the manifold is finite by an argument analogous to that in item $(2)$ (Bonnet--Myers theorem~\ref{Myers--Bonnet}).  
\end{enumerate}

\section{The distance function from a point}\label{The distance function from a point}

Let $(M,g)$ be a complete, $n$--dimensional Riemannian manifold. We now consider $r:M\to\R$ to be the distance function $d_p$ from a point $p\in M$, which is well known to be smooth in the open set $D_p=M\setminus\bigl(\{p\}\cup \mathrm{Cut}_p\bigr)$, where $\mathrm{Cut}_p$ is the cut~locus of $p$ (we recall that the {\em squared} distance function from $p$, given by $d_p^2$ is smooth also at $p$ and its Hessian is positive definite in a neighborhood of such point). Notice that the polar coordinates $(\theta, r)$ defined on an open $U\subseteq D_p$ are then adapted to $r$ and every $U_\rho$ is the part of the geodesic sphere of radius $\rho>0$, centered at $p$, belonging to $U$. 
Then, the metric tensor can be written as
$$
g =  g_{\alpha \beta} \, d\theta^\alpha\! \otimes d\theta^\beta+dr\otimes dr.
$$
As we said, the coordinate vector fields $ \partial_{\theta^{1}}, \dots,\partial_{\theta^{n-1}},\partial_{r}$ are Jacobi vector fields for $r$, while $\partial_r$ is a parallel vector field for $r$.\\
We recall how such polar coordinates are defined. Let $\SSSS^{n-1}_p\subseteq T_pM$ be the unit sphere with respect to $g_p$. 
From now on, we fix a linear isometry $\rchi$ from $(T_pM,g_{p})$ to $(\R^{n},\langle\cdot\,,\cdot\rangle)$ which maps the vectors of a chosen orthonormal basis $\mathcal{B}=\{e_{1},\dots,e_{n}\}$ of $T_{p}M$ to the vectors of the canonical orthonormal basis of $\R^{n}$ and we consider a chart  $(V,\theta)$ of $\SSSS^{n-1}_p$, obtained by identifying $\SSSS^{n-1}_p$ with the $(n-1)$--dimensional standard unit sphere $\SSSS^{n-1}$ (whose canonical metric is denoted by $g_{\SSSS^{n-1}}$) through the isometry $\rchi$. For some $\varepsilon>0$ the open set 
$$
U=\bigl\{\exp_p(r\xi)\,:\, \text{$r\in(0,\varepsilon)$ and $\xi\in V$}\bigr\}
$$
is contained in $D_p$, therefore the map
$$
\phi:(\xi,r)\in V\times(0,\varepsilon)\subseteq\SSSS^{n-1}_p\times\R^+\longmapsto\exp_{p}(r\xi)\in U\subseteq M
$$
is a diffeomorphism. This, in turn, implies that the map
$$
\phi\circ (\theta^{-1},\id):(\theta^{1},\dots,\theta^{n-1},r)\in\theta (V)\times(0,\varepsilon)\longmapsto\exp_{p}\big(r\xi(\theta)\big)\in U
$$
is also a diffeomorphism whose inverse $\varphi$ provides polar coordinates on $U$.\\
At the same time, we can consider on the open set $U$ the chart $x$ given by the normal coordinates with respect to the fixed orthonormal basis $\mathcal{B}$ of $T_{p}M$. Notice that the transition map from $\varphi$ to $x$ is given by 
$$
x\circ\varphi^{-1}: (\theta,r)\in \theta (V)\times(0,\varepsilon)\subseteq\R^{n-1}\times\R^+\longmapsto (r\xi^{1}(\theta),\dots,r\xi^{n}(\theta))\in \R^{n}.
$$

In order to get information from the previous differential equations and inequalities, we first analyze the behavior of the metric and of the Hessian of the distance function $r=d_p$ along any (minimal) geodesic $r\mapsto \gamma(r)=\exp_{p}(r\xi)$ in $U$ (which is an integral curve for $\partial_r$), as $r\to 0^+$.

\begin{proposizione}\label{ppp} The following limits hold along any unit--speed geodesic $r\mapsto \gamma(r)=\exp_{p}(r\xi)$, for $\xi\in\SSSS^{n-1}_p$,

\medskip
 
\begin{enumerate}
\item  \ \\
\vskip-1,6cm
$$
\lim_{r \to 0^{+}} g_{\alpha\beta}=\lim_{r \to 0^{+}} g(\partial_{\theta^\alpha},\partial_{\theta^\beta})=0
$$
for each $\alpha,\beta\in\{1,\dots,n-1\}$. Consequently
$$
\lim_{r \to 0^{+}} g(J_1,J_2)=0
$$
for every pair $J_{1},J_{2}$ of Jacobi fields for the distance function $r$, pointwise orthogonal to $\partial_{r}$, being $J_1,J_2$ pointwise linear combinations of the coordinate fields $\partial_{\theta^\alpha}$ with coefficients independent of $r$.\\ 

\item \ \\
\vskip-1,6cm
$$
\lim_{r \to 0^{+}} \Big(\mathrm{Hess}\,r (\partial_{\theta^\alpha},\partial_{\theta^\beta})-\frac{g_{\alpha\beta}}{r}\Big)=0
$$
for each $\alpha,\beta\in\{1,\dots,n-1\}$. Consequently
\begin{equation}\label{hesslim1}
\lim_{r \to 0^{+}} \Big(\mathrm{Hess}\,r (J_{1},J_{2})-\frac{g(J_{1},J_{2})}{r}\Big)=0
\end{equation}
for each pair $J_{1},J_{2}$ of Jacobi fields for the distance function $r$, pointwise orthogonal to $\partial_{r}$.\\ 

\item \ \\
\vskip-1,6cm
$$
\lim_{r \to 0^{+}} \Big(\frac{\mathrm{Hess}\,r(J,J)}{g(J,J)}-\frac{1}{r}\Big)=0\label{hesslim2}
$$
for every Jacobi field $J$ for the distance function $r$, which is pointwise orthogonal to $\partial_{r}$ and never zero on the geodesic curve $\gamma$.\\ 

\item \ \\
\vskip-1,6cm
$$
\lim_{r \to 0^{+}}\Big( \frac{\HHH}{n-1}-\frac{1}{r}\Big)=0.
$$
\\

\item \ \\
\vskip-1,6cm
$$
\lim_{r \to 0^{+}} \Big(\mathrm{Hess}\,r(E,E) -\frac{1}{r}\Big)=0
$$
for every parallel field $E$ for the distance function $r$, with unit norm and pointwise orthogonal to $\partial_{r}$.
\end{enumerate}
\end{proposizione}
\begin{proof}
Let us denote by $\{\partial_{1},\dots,\partial_{n}\}$ the coordinate vector fields associated with the normal coordinates chart $(U,x)$ and by $\{\partial_{\theta^{1}},\dots,\partial_{\theta^{n-1}},\partial_{r}\}$ those induced by the polar coordinates chart $(U,\varphi)$, where both charts are as above. Moreover, 
let $\xi=\xi(\theta)$. The first limit is a direct consequence of the fact that 
$$
g_{\alpha \beta}(\gamma(r))=O(r^2),
$$
since there holds
\begin{equation}
g_{\alpha \beta}(\gamma(r))=g_{ij}(\gamma(r))\frac{\partial(x^{i}\circ\varphi^{-1})}{\partial\theta^{\alpha}}(\theta,r)\frac{\partial(x^{j}\circ\varphi^{-1})}{\partial\theta^{\beta}}(\theta,r)
=\,r^{2}g_{ij}(\gamma(r))\frac{\partial \xi^{i}}{\partial\theta^{\alpha}}(\theta)\frac{\partial \xi^{j}}{\partial\theta^{\beta}}(\theta),\label{eq710}
\end{equation}
for every $\alpha,\beta\in\{1,\dots,n-1\}$.

For the second limit, we start by observing that
\begin{align*}
\mathrm{Hess}\,r\big(\partial_{\theta^{\alpha}}|_{\gamma(r)},\partial_{\theta^{\beta}}|_{\gamma(r)}\big)
=&\,\,\mathrm{Hess}\,r\Bigl(r  \frac{\partial\xi^{i}}{\partial\theta^{\alpha}}(\theta)\partial_{i}|_{\gamma(r)},r\frac{\partial\xi^{j}}{\partial\theta^{\beta}}(\theta)\partial_{j}|_{\gamma(r)}\Bigr)\\
=&\, r^{2}\frac{\partial\xi^{i}}{\partial\theta^{\alpha}}(\theta)\frac{\partial\xi^{j}}{\partial\theta^{\beta}}(\theta)\,\mathrm{Hess}\,r\big(\partial_{i}|_{\gamma(r)},\partial_{j}|_{\gamma(r)}\big)\\
=&\,r\frac{\partial\xi^{i}}{\partial\theta^{\alpha}}(\theta)\frac{\partial\xi^{j}}{\partial\theta^{\beta}}(\theta)\Big[\delta_{ij}-\xi^{i}(\theta)\xi^{j}(\theta)-r\Gamma^{k}_{ij}(\gamma(r))\xi^{k}(\theta)\Big]\\
=&\, r\frac{\partial\xi^{i}}{\partial\theta^{\alpha}}(\theta)\frac{\partial\xi^{j}}{\partial\theta^{\beta}}(\theta)\Big[\delta_{ij}-r \Gamma^{k}_{ij}(\gamma(r))\xi^{k}(\theta)\Big],
\end{align*}
where we used the identities
$$
\mathrm{Hess}\,r(\partial_{i}, \partial_{j})=\frac{\delta_{ij}}{r}-\frac{x^{i}x^{j}}{r^{3}}- \Gamma^{k}_{ij}\frac{x^{k}}{r}\qquad\text{ and }\qquad\xi^{i}(\theta) \frac{\partial\xi^{i}}{\partial\theta^{\alpha}}(\theta)=0.
$$
Hence, by equality~\eqref{eq710}, we have
$$
\mathrm{Hess}\,r(\partial_{\theta^{\alpha}}|_{\gamma(r)},\partial_{\theta^{\beta}}|_{\gamma(r)})-\frac{1}{r}g_{\alpha \beta}(\gamma(r))=r\frac{\partial\xi^{i}}{\partial\theta^{\alpha}}(\theta)\frac{\partial\xi^{j}}{\partial\theta^{\beta}}(\theta)\Big[\delta_{ij}-g_{ij}(\gamma(r))-r \Gamma^{k}_{ij}(\gamma(r))\xi^{k}(\theta)\Big]
$$
and by the properties of the normal coordinates, we obtain the second limit.

Let us consider a nonzero linear combination $J$ of $\partial_{\theta^{1}},\dots,\partial_{\theta^{n-1}}$, that is $J=a^{\alpha }\partial_{\theta^{\alpha}} $ with some coefficient $a^\alpha\neq 0$.
Then, we have 
\begin{align}
 \frac{\mathrm{Hess}\,r(J,J)}{g(J,J)}\,\bigg|_{\gamma(r)}-\frac{1}{r}&=\frac{1}{r}\Bigg(\frac{  a^{\alpha}a^{\beta}\frac{\partial\xi^{i}}{\partial\theta^{\alpha}}(\theta)\frac{\partial\xi^{j}}{\partial\theta^{\beta}}(\theta)\big[\delta_{ij}-r \Gamma^{k}_{ij}(\gamma(r))\xi^{k}(\theta)\big]}{a^{\alpha}a^{\beta}\frac{\partial\xi^{i}}{\partial\theta^{\alpha}}(\theta)\frac{\partial\xi^{j}}{\partial\theta^{\beta}}(\theta) g_{ij}(\gamma(r))}-1\Bigg)\\
&=\frac{ a^{\alpha}a^{\beta}\frac{\partial\xi^{i}}{\partial\theta^{\alpha}}(\theta)\frac{\partial\xi^{j}}{\partial\theta^{\beta}}(\theta)\,\frac{1}{r} [\delta_{ij}- g_{ij}(\gamma(r))-r \Gamma^{k}_{ij}(\gamma(r))\xi^{k}(\theta)]\,}{a^{\alpha}a^{\beta}\frac{\partial\xi^{i}}{\partial\theta^{\alpha}}(\theta)\frac{\partial\xi^{j}}{\partial\theta^{\beta}}(\theta) g_{ij}(\gamma(r))}.
\end{align}
Now, we observe that l'H\^opital's rule, together with the properties of the normal coordinates, implies
\begin{equation}\label{eqf:0}
\lim_{r \to 0^{+}}\frac{1}{r}\Big[g_{ij}(\gamma(r))-\delta_{ij}\Big]=\lim_{r \to 0^{+}}\xi^{k}(\theta)\frac{\partial g_{ij}}{\partial x^{k}}(\gamma(r))=0.
\end{equation}
Then, the third limit follows by observing that 
\begin{equation}\label{ffeq1}
\lim_{r \to 0^{+}}a^{\alpha}a^{\beta}\frac{\partial\xi^{i}}{\partial\theta^{\alpha}}(\theta)\frac{\partial\xi^{j}}{\partial\theta^{\beta}}(\theta) g_{ij}(\gamma(r))=
a^{\alpha}a^{\beta}\frac{\partial\xi^{i}}{\partial\theta^{\alpha}}(\theta)\frac{\partial\xi^{j}}{\partial\theta^{\beta}}(\theta)\delta_{ij}=a^{\alpha}a^{\beta} g_{\alpha\beta}^{\SSSS^{n-1}_p}(\xi)>0.
\end{equation}

By the previous computations, we have
\begin{eqnarray*}
\HHH({\gamma(r)})-\frac{n-1}{r}&=&\tr\mathrm S({\gamma(r)})-\frac{n-1}{r}\\
&=&g^{\alpha \beta}(\gamma(r))\Big(\mathrm{Hess}\,r (\partial_{\theta^{\alpha}}|_{\gamma(r)},\partial_{\theta^{\beta}}|_{\gamma(r)})-\frac{1}{r}g_{\alpha \beta}(\gamma(r))\Big)\\
&=&r g^{\alpha \beta}(\gamma(r))\frac{\partial \xi^{i}}{\partial\theta^{\alpha}}(\theta)\frac{\partial \xi^{j}}{\partial\theta^{\beta}}(\theta)\Big[\delta_{ij}-g_{ij}(\gamma(r))-r\Gamma^{k}_{ij}(\gamma(r))\xi^{k}(\theta)\Big].
\end{eqnarray*}
Since the identities~\eqref{eq710} and~\eqref{ffeq1} imply 
$$
g_{\alpha \beta}(\gamma(r))=r^{2}a_{\alpha\beta}(r,\theta)\quad\quad \text{and}\quad\quad \lim_{r \to 0^{+}}a_{\alpha\beta}(r,\theta)=g_{\alpha\beta}^{\SSSS^{n-1}_p}(\xi),
$$
we get that
$$\HHH({\gamma(r)})-\frac{n-1}{r}=a^{\alpha \beta}(r,\theta)\frac{\partial \xi^{i}}{\partial\theta^{\alpha}}(\theta)\frac{\partial \xi^{j}}{\partial\theta^{\beta}}(\theta)\,\frac{1}{r}\Big[\delta_{ij}-g_{ij}(\gamma(r))-r \Gamma^{k}_{ij}(\gamma(r))\xi^{k}(\theta)\Big], $$
where each $a^{\alpha \beta}(r,\theta)$ converges to $g^{\alpha\beta}_{\SSSS^{n-1}_p}(\xi)$ as $r\to 0^+$.
Then, the fourth limit follows from what was said earlier.\\
Concerning the last limit, let $E$ be a unit parallel vector field for the distance function $r$, pointwise orthogonal to $\partial_{r}$, then
$$
E=b^{\alpha}\partial_{\theta^{\alpha}}= r b^{\alpha} \frac{\partial\xi^{i}}{\partial\theta^{\alpha}}\, \partial_{i}.
$$
Now, we notice that, since $E(\gamma(r))$ is a parallel vector field along the curve $\gamma|_{(0,\varepsilon)}$, the functions $r b^{\alpha}(\gamma(r))(\partial\xi^{i}/\partial\theta^{\alpha})(\theta)$ converge as $r \rightarrow 0^{+}$, therefore, we have
\begin{align*}
\Big(\mathrm{Hess}\,r\,(E,E) -\frac{1}{r}\Big)&\,\bigg|_{\gamma(r)}=b^{\alpha}(\gamma(r))b^{\beta}(\gamma(r))\Big(\mathrm{Hess}\,r(\partial_{\theta^{\alpha}}|_{\gamma(r)},\partial_{\theta^{\beta}}|_{\gamma(r)})-\frac{1}{r}g_{\alpha \beta}(\gamma(r))\Big)\\
=&\,r^2b^{\alpha}(\gamma(r))b^{\beta}(\gamma(r))\frac{\partial\xi^{i}}{\partial\theta^{\alpha}}(\theta)\frac{\partial\xi^{j}}{\partial\theta^{\beta}}(\theta)\,\frac{1}{r}\Big[\delta_{ij}-g_{ij}(\gamma(r))-r \Gamma^{k}_{ij}(\gamma(r))\xi^{k}(\theta)\Big],
\end{align*}
which yields the desired limit.
\end{proof}

Let us introduce the last ``analytical'' tool, which is the following general result for differential inequalities.

\begin{proposizione}[Riccati comparison principle]\label{riccati}
Given two $C^1$ functions $\rho_1,\rho_2:(0,b)\rightarrow \R$ such that 
\begin {enumerate}
\item $\rho_{1}'+\rho_{1}^{2}\leqslant\rho_{2}'+\rho_{2}^{2}$\\

\item $\lim_{t \to 0^{+}} \bigl(\rho_i(t)-\frac{1}{t}\bigr)=0$
\end{enumerate}

\smallskip

\noindent for $i\in\{1,2\}$, we have $\rho_{2}(t)\geqslant \rho_{1}(t)$ for every $t\in(0,b)$.
\end{proposizione}

From this proposition it clearly follows that if two functions $\rho_1,\rho_2:(0,b)\rightarrow \R$ satisfy
$$
\rho_{1}'+\rho_{1}^{2}=\rho_{2}'+\rho_{2}^{2}\qquad\text{ and }\qquad\lim_{t \to 0^{+}} \Bigl(\rho_i(t)-\frac{1}{t}\Bigr)=0,
$$
for $i\in\{1,2\}$, then $\rho_{1}=\rho_{2}$.

\begin{proof}
We choose $t_{0}\in(0,b)$ and define
$$
f(t)=\int ^{t}_{t_{0}} \bigl(\rho_{2}(\tau)+\rho_{1}(\tau)\bigr)\,d\tau
$$
for any $ t\in (0,b)$.
Since the first hypothesis leads to
$$
\frac {d}{dt} \big((\rho_{2}-\rho_{1})e^{f}\big)=(\rho_{2}'-\rho_{1}'+\rho_{2}^{2}-\rho_{1}^{2})e^{f}\geqslant0,
$$
the function $(\rho_{2}-\rho_{1})e^{f}$ is nondecreasing in $(0,b)$.
It then follows that
\begin{equation}\label{eq600}
\rho_{2}(t_2)-\rho_{1}(t_2)\geqslant \bigl(\rho_{2}(t_{1})-\rho_{1}(t_{1})\bigr)e^{f(t_{1})-f(t_2)},
\end{equation}
for every $0<t_{1}<t_2<b$.
Now, by the second hypothesis,
$\rho_i(t)=t^{-1}+o_t(1)$ for each $i\in\{1,2\}$, where the expression $o_t(1)$ denotes any function converging to zero as $t\to 0^+$. Then, we observe that 
\begin{equation}\label{eq601}
\rho_{2}(t)-\rho_{1}(t)=o_t(1)
\end{equation}
and 
\begin{equation}\label{eq601bis}
f(t)=\int ^{t}_{t_{0}} \bigl(\rho_{2}(\tau)+\rho_{1}(\tau)\bigr)\,d\tau=2\log\Big(\frac{t}{t_0}\Big)+ O_t(1),
\end{equation}
where the expression $O_t(1)$ is used to denote any function bounded in a right neighborhood of $0$. 
Therefore, the function $e^{f(t)}$ converges to zero as $t\to 0^+$. By using this result along with equality~\eqref{eq601} in inequality~\eqref{eq600}, we get
$$
\rho_{2}(t_2)-\rho_{1}(t_2)\geqslant \lim_{t_1\to0^+}\bigl(\rho_{2}(t_{1})-\rho_{1}(t_{1})\bigr)e^{f(t_{1})-f(t_2)}=0,
$$
which yields the desired inequality.
\end{proof}

Given a real constant $k$, we define the function ${\mathrm{sn}}_{k}:\R\to\R$ as
\begin{equation}\label{SNK}
{\mathrm{sn}}_{k}(t) =
\begin{cases}
{\displaystyle{\sin \big(\sqrt{k\,}t\big)/\sqrt{k\,}}} &  \text{ for $k>0$} \\
{\displaystyle{t}}  & \text{ for $k=0$}\\
{\displaystyle{\sinh \big(\sqrt{-k}t\big)/\sqrt{-k}}}  &  \text{ for $k<0$}
\end{cases}
\end{equation}
which can be easily seen to satisfy the ordinary differential equation
$$
f''+kf=0
$$
with $f(0)=0$ and $f'(0)=1$. Then, adopting the convention that ${\pi}/{\sqrt{k\,}}=+\infty$ when $k\leqslant 0$, by direct check, the function ${{\mathrm{sn}}_{k}'(t)}/{{\mathrm{sn}}_{k}(t)}$ on the open interval $(0,{\pi}/{\sqrt{k\,}})$ satisfies the ordinary differential equation
$$
f'+f^{2}=-k
$$
and
$$
\lim_{t \to 0^{+}} \Bigl(\frac{{\mathrm{sn}}_{k}'(t)}{{\mathrm{sn}}_{k}(t)}-\frac{1}{t}\Bigr)=0.
$$

Applying then the Riccati comparison principle, we obtain the following inequalities along any minimal geodesic $\gamma$ from $p\in M$, as above:

\begin{enumerate}
\item If $(M,g)$ has sectional curvature $\Sec\leqslant K$, then from formulas~\eqref{eq810} and~\eqref{eq:fo2}, we have
$$
\partial_{r}\Big( \frac{\mathrm{Hess}\,r(J,J)}{g(J,J)}\Big)+\Big(\frac{\mathrm{Hess}\ r(J,J)}{g(J,J)}\Big)^{2} \geqslant-K=\Big(\frac{{\mathrm{sn}}_K'}{{\mathrm{sn}}_K}\Big)^\prime+\Big(\frac{{\mathrm{sn}}_K'}{{\mathrm{sn}}_K}\Big)^{2}, 
$$
for every Jacobi vector field $J$ for the distance function $r$, which does not vanish along $\gamma$ and is at every point a linear combination of $\partial_{\theta^{1}},\dots,\partial_{\theta^{n-1}}$ (hence at every point orthogonal to $\partial_r$).
Recalling the third limit in Proposition~\ref{ppp}, we can then apply Proposition~\ref{riccati} and conclude that
\begin{equation}\label{Jaceq-1}
\frac{\mathrm{Hess} \,r(J,J)}{g(J,J)}(r)\geqslant \frac{{\mathrm{sn}}_K'(r)}{{\mathrm{sn}}_K(r)},
\end{equation}
implying
\begin{equation}\label{Jaceq}
{\mathrm{Hess}}\,r\geqslant\frac{{\mathrm{sn}}_K'(r)}{{\mathrm{sn}}_K(r)}\,g
\end{equation}
on $T_qS_r(p)$, where $q=\exp_p(rv)$ and $S_r(p)=\exp_p(r\SSSS^{n-1}_p)$, for every $r\in (0,\varepsilon)$ in the domain of the function ${\mathrm{sn}}_K'/{\mathrm{sn}}_K$.
\item If $(M,g)$ has sectional curvature $\Sec\geqslant k$, then from the formula~\eqref{eq510} we have
$$
\partial_{r} \mathrm{Hess}\,r(E,E)+[\mathrm{Hess}\,r(E,E)]^{2}\leqslant - k=\Big(\frac{{\mathrm{sn}}_k'}{{\mathrm{sn}}_k}\Big)^\prime+\Big(\frac{{\mathrm{sn}}_k'}{{\mathrm{sn}}_k}\Big)^{2}, 
$$
for any parallel vector field $E$ for the distance function $r$ with unit norm and pointwise orthogonal to $\partial_{r}$. Then, by the fifth limit in Proposition~\ref{ppp} and Proposition~\ref{riccati}, we get
\begin{equation}\label{Pareq-1}
\mathrm{Hess}\,r(E,E)(r)\leqslant \frac{{\mathrm{sn}}_k'(r)}{{\mathrm{sn}}_k(r)},
\end{equation}
implying
\begin{equation}\label{Pareq}
{\mathrm{Hess}}\,r\leqslant\frac{{\mathrm{sn}}_k'(r)}{{\mathrm{sn}}_k(r)}\,g
\end{equation}
on $T_qS_r(p)$, where $q=\exp_p(rv)$ and $S_r(p)=\exp_p(r\SSSS^{n-1}_p)$, for every $r\in (0,\varepsilon)$ in the domain of the function ${\mathrm{sn}}_k'/{\mathrm{sn}}_k$.
\item If $(M,g)$ satisfies $\Ric\geqslant k(n-1)g$ for a real value $k$, from the second point of Proposition~\ref{eqS}, we get 
$$
\partial_{r}\Big(\frac{\HHH}{n-1}\Big)+\Big(\frac{\HHH}{n-1}\Big)^{2}\leqslant-k.
$$
Then, by the fourth limit in Proposition~\ref{ppp}, Proposition~\ref{riccati} implies
\begin{equation}\label{trSeq}
\Delta r=\HHH(r)\leqslant (n-1)\frac{{\mathrm{sn}}_k'(r)}{{\mathrm{sn}}_k(r)},
\end{equation}
for every $r\in (0,\varepsilon)$ in the domain of the function ${\mathrm{sn}}_k'/{\mathrm{sn}}_k$.
\end{enumerate}
The cases with $K$ or $k$ equal to zero are particularly relevant:
\begin{enumerate}
\item If $(M,g)$ has nonpositive sectional curvature $\Sec\leqslant0$, then
\begin{equation}\label{Jaceqbis}
{\mathrm{Hess}}\,r\geqslant\frac{g}{r}
\end{equation}
on $T_qS_r(p)$, with $q=\exp_p(rv)$ and $r\in (0,\varepsilon)$.\\

\item If $(M,g)$ has nonnegative sectional curvature $\Sec\geqslant0$, then
\begin{equation}
{\mathrm{Hess}}\,r\leqslant\frac{g}{r}
\end{equation}
on $T_qS_r(p)$, with $q=\exp_p(rv)$ and $r\in (0,\varepsilon)$.\\
\item If $(M,g)$ satisfies $\Ric\geqslant0$, then
\begin{equation}\label{Hconfr}
\Delta r=\HHH(r)\leqslant\frac{n-1}{r}
\end{equation}
for every $r\in (0,\varepsilon)$.\\
\end{enumerate}
It follows that, if $(M,g)$ has constant curvature $k$, there holds
\begin{equation}\label{Jconst}
{\mathrm{Hess}}\,r=\frac{{\mathrm{sn}}_k'(r)}{{\mathrm{sn}}_k(r)}\,g
\end{equation}
on $T_qS_r(p)$, with  $q=\exp_p(rv)$ and ${\mathrm{Hess}}\,r(\partial_r,\cdot)=0$, for every $r\in (0,\min\{\varepsilon,\pi/\sqrt{k\,}\})$ (setting $\pi/\sqrt{k\,}=+\infty$, if $k\leqslant0$).

Being able to cover all $D_{p}\setminus\{p\}$ by polar coordinate systems as above, choosing appropriate open sets $U$, we conclude that these inequalities hold along all geodesics $\gamma_\xi$, for every $\xi\in \SSSS^{n-1}_{p}$ and for every $r\in (0,c(\xi))$ (where $c(\xi)$ is the first value such that $\gamma_{\xi}$ reaches the cut~locus of $p$ with $c(\xi)=+\infty$ if there is no cut point in that direction) within the domain of the corresponding functions, ${\mathrm{sn}}_K'/{\mathrm{sn}}_K$ or ${\mathrm{sn}}_k'/{\mathrm{sn}}_k$.

\medskip

The {\em space form} $(\mathbb{M}_k^n,\mathrm{g}_k^n)$ is given by $(\R^n,g_{\mathrm{can}})$, when $k=0$, the sphere $(\SSSS^n_{1/\sqrt{k\,}},g_{\mathrm{can}})$, when $k>0$ and the hyperbolic space $\HH^n$ with its canonical metric rescaled by the factor $-1/k$, when $k<0$.\\
For any $p\in\mathbb{M}_k^n$, the open set obtained from $\mathbb{M}_k^n$ by removing $p$ and the cut~locus of $p$ is isometric, with its natural metric, to the following warped product:
\begin{equation}\label{baseSF}
(M^{n}_{k},g^{n}_{k})=\big(\SSSS^{n-1}\times(0,\pi/\sqrt{k\,}\,),{\mathrm{sn}}_{k}^{2}(r)\,g_{\SSSS^{n-1}}+dr\otimes dr\big)
\end{equation}
with the convention 
\begin{equation}\label{convention}
{\pi}/{\sqrt{k\,}}=
\begin{cases}
{\pi}/{\sqrt{k\,}} &  \text{ for $k>0$} \\
+\infty  & \text{ for $k\leqslant 0$}
\end{cases}
\end{equation}

\begin{osservazione}\label{remarkHk}
We observe that the right-hand sides of inequalities~\eqref{Jaceq},~\eqref{Pareq} and~\eqref{trSeq} are the analogous quantities to the ones on the left-hand sides, for the space forms with constant curvature. In particular, if we denote by $\HHH_k(r)$ the mean curvature of the geodesic sphere of radius $r$ of the space form with constant curvature $k\in\R$, we have the following {\em mean curvature comparison inequality}
\begin{equation}\label{trSeqbis}
\HHH(r)\leqslant \HHH_k(r)\,,
\end{equation}
along any minimal geodesic $\gamma$ from $p\in M$, as above, under the assumption that $(M,g)$ satisfies $\Ric\geqslant k(n-1)g$.
\end{osservazione}

\begin{osservazione}\label{OSSdistgen}
Along the same lines, based on the Riccati comparison principle, analogous conclusions (suitably modified) can also be drawn if $r$ is the distance function from a submanifold, or a general distance function (see~\cite[Section~3]{agray}).
\end{osservazione}

\section{Applications}\label{curvtop}

In this section we are going to apply the ``analytical'' tools (differential equations and inequalities) that we discussed in the previous sections, to show ``alternative'' proofs of some ``classical'' results about the relationships between curvature and topology, which are usually proved by means of the Jacobi vector fields and the second variation formula of the length/energy functional for geodesics. The ``standard'' proofs can be found in~\cite{gahula}, for instance. We invite the reader to compare such proofs in order to evaluate the differences in the complexity of the two approaches.

\medskip

\noindent{\textbf{Riemannian manifolds with constant sectional curvature}\smallskip\\
The following proposition is the local version of the well known ``classification theorem'', asserting that any complete Riemannian manifold with constant curvature is the quotient of a space form (see~\cite[Section~3.F]{gahula}, for instance). 

\begin{proposizione}\label{constsec}
If the $n$--dimensional Riemannian manifold $(M,g)$ has constant sectional curvature $k$, then for every $p\in M$ there exists an open neighborhood of $p$ isometric to an open subset of the $n$--dimensional space form $(\mathbb{M}^n_k,\mathrm{g}_k^n)$ of constant curvature $k\in\R$.
\end{proposizione}
\begin{proof}[Proof via the distance functions]\ \\
With the notation of the previous section, considering polar coordinates $(r,\theta)$ in a neighborhood of $p\in M$, we have $r=d_p$ and we concluded in equation~\eqref{Jconst} that
\begin{equation}\label{Jconst2}
{\mathrm{Hess}}\,r=\frac{{\mathrm{sn}}_k'(r)}{{\mathrm{sn}}_k(r)}\,g
\end{equation}
on $T_qS_r(p)$, for every $q=\exp_p(rv)\in S_r(p)$ and $v=v(\theta)$.\\
From the expression for the metric along the geodesic $r\mapsto\gamma(r)=\exp_{p}(rv)$ in the proof of the first point of Proposition~\ref{ppp}, we have
\begin{equation}\label{eq:fff0}
\lim_{r \to 0^{+}}\frac{g_{\alpha \beta}(\gamma(r))}{{\mathrm{sn}}^{2}_{k}(r)}=g_{\alpha \beta}^{\SSSS^{n-1}_p}\!(v)\,,
\end{equation}
for every $\alpha,\beta\in\{1,\dots,n-1\}$.\\
From the first equation of Proposition~\ref{eqriempol}, we have
$$
\partial_{r} g_{\alpha \beta}(\gamma(r))=2\mathrm{Hess}\,r(\partial_{\theta^{\alpha}},\partial_{\theta^{\beta}})=2\,\frac{{\mathrm{sn}}_{k}'(r)}{{\mathrm{sn}}_{k}(r)} g_{\alpha \beta}(\gamma(r))
$$
therefore,
$$
\frac{d\,}{dr} \frac{g_{\alpha \beta}(\gamma(r))}{{\mathrm{sn}}^2_k(r)}=\frac{{\mathrm{sn}}^2_k(r)\partial_rg_{\alpha \beta}(\gamma(r))-2{\mathrm{sn}}_k(r){\mathrm{sn}}'_k(r) g_{\alpha \beta}(\gamma(r))}{{\mathrm{sn}}^4_k(r)}= 0\,,
$$
that is, the function $r\mapsto\frac{g_{\alpha \beta}(\gamma(r))}{{\mathrm{sn}}^2_k(r)}$ is constant and we have
$$
g_{\alpha \beta}(\gamma(r))=g_{\alpha \beta}^{\SSSS^{n-1}_p}\!(v)\,{\mathrm{sn}}^{2}_{k}(r)\,,
$$
by the limit~\eqref{eq:fff0}. Hence,
$$
g={\mathrm{sn}}_{k}^{2}(r)\,g_{\SSSS^{n-1}_p}+dr\otimes dr\,,
$$
so $(M,g)$ is locally isometric to $(\mathbb{M}^n_k, \mathrm{g}_k^n)$, by the discussion immediately before Remark~\ref{remarkHk}.
\end{proof}

\medskip

\noindent{\textbf{The Bonnet--Myers theorem and Synge's lemma}\smallskip\\
One of the most classical results in Riemannian geometry is the following Bonnet--Myers ``diameter estimate''.

\begin{teorema}[Bonnet--Myers theorem]\label{Myers--Bonnet}
Let $(M,g)$ be an $n$--dimensional complete Riemannian manifold such that $\Ric \geqslant k(n-1)g$, for some $k>0$. Then, the diameter of $M$ is bounded above by ${\pi}/{\sqrt{k\,}}$.
\end{teorema}
\begin{proof}[Proof via the distance functions]\ \\
Assume for contradiction that there exists a minimizing geodesic $\gamma : [0,L] \to M$ of unit speed starting from $p\in M$, with $L>{\pi}/{\sqrt{k\,}}$. The distance function $d_p$ is smooth along $\gamma((0,L))$ and, for every $r\in (0,L)$, the geodesic sphere $S_r(p)$ is a hypersurface in a neighborhood of every point $\gamma(r)$. By the inequality~\eqref{trSeq}, along the image of $\gamma$ the mean curvature $\HHH(r)$ of the geodesic sphere $S_r(p)$ would satisfy
$$
\HHH(r)\leqslant(n-1) \sqrt{k\,}\cot(\sqrt{k\,} r)\,.
$$
Since $L>{\pi}/{\sqrt{k\,}}$, this would lead to a contradiction since $\cot(\sqrt{k\,} r) \to - \infty$ as $r\to({\pi}/\sqrt{k\,}\,)^-$.
\end{proof}

For completeness, we mention that this diameter estimate is optimal, as the standard $n$--dimensional sphere $\SSSS^{n}_{1/\sqrt{k\,}}$ of radius $1/\sqrt{k\,}$ has Ricci tensor equal to $(n-1)kg$ and diameter $\pi/\sqrt{k\,}$. Moreover, {\em Cheng's maximal diameter theorem} (see~\cite{petersen2}) says that if $(M, g)$ satisfies the assumptions  of the Bonnet--Myers theorem and has (maximal) diameter equal to $\pi/\sqrt{k\,}$, then $(M, g)$ is isometric to $\SSSS^{n}_{1/\sqrt{k\,}}$ with its canonical metric. Finally, this estimate implies (by passing to the universal covering of the manifold) that the manifold has finite fundamental group, besides being compact.
\medskip

Although Synge's lemma below is clearly a special (weaker) case of the Bonnet--Myers theorem, we nevertheless prove it for historical reasons (it actually came before, obtained by Synge as a consequence of his second variation formula for the length functional~\cite{synge1}) and since in this way we take the opportunity to show a slightly different line, still based on the properties of the distance functions, that can be adapted to prove Theorem~\ref{Myers--Bonnet} also.

\begin{lemma}[Synge's lemma]\label{Synge}
Let $(M,g)$ be a complete Riemannian manifold with sectional curvature $\Sec(\pi)\geqslant k>0$ for every $2$--plane $\pi\subseteq T_pM$ and every $p\in M$. Then $M$ is compact, with diameter less than or equal to ${\pi}/{\sqrt{k\,}}$.
\end{lemma}
\begin{proof}[Proof via the distance functions]\ \\
We follow the same proof strategy as in the Bonnet--Myers theorem using the fundamental equations, but reasoning on the maximal eigenvalue $\mu_{\max}$ of the matrix ${\mathrm S}_\alpha^\beta$ in polar coordinates centered at $p$, which is the maximal eigenvalue of the shape operator of the geodesic spheres centered at $p$.\\
Assume for contradiction that there exists a minimizing geodesic $\gamma : [0,L] \to M$ of unit speed starting from $p\in M$, with $L>{\pi}/{\sqrt{k\,}}$. Then, we have seen that along such geodesic, the differential inequality~\eqref{eq333} holds,
$$
\partial_{r}\mu_{\max}(r)\leqslant -\mu_{\max}^2(r)-{\mathrm{MinSec}}\,,
$$
for almost every $r\in(0,L)$ (the function $\mu_{\max}$ is locally absolutely continuous), so under our assumptions,
\begin{equation}\label{eq33333}
\partial_{r}\mu_{\max}(r)+\mu_{\max}^2(r)\leqslant -k\,.
\end{equation}
Since ${\mathrm S}_\alpha^\beta=\mathrm{h}_{\alpha \lambda} g^{\lambda \beta}$ and $\mathrm{h}_{\alpha \beta}= \mathrm{Hess}\,r(\partial_{\theta^{\alpha}},\partial_{\theta^{\beta}})$, by performing computations similar to those carried out to obtain the second and the third limits in Proposition~\ref{ppp}, it follows that
\begin{equation}\label{eq11111}
{\mathrm S}_\alpha^\beta=\frac{\delta_{\alpha}^{\beta}}{r}+o(1)\,,
\end{equation}
as $r\to0^+$. Consequently, we obtain
$$
\mu_{\max}(r)-\frac{1}{r}=o(1)\,,
$$
as $r\to 0^+$.\\
Applying then the Riccati comparison principle, Proposition~\ref{riccati} (which can be easily seen to work even if the functions are only locally absolutely continuous and the inequality in the statement holds only almost everywhere), we conclude that 
$$
\mu_{\max}(r)\leqslant\frac{{\mathrm{sn}}'_{k}(r)}{{\mathrm{sn}}_{k}(r)}\,,
$$
for every $r\in(0,\pi/\sqrt{k\,}\,)$, as $L>{\pi}/{\sqrt{k\,}}$. 
This would lead to a contradiction, since $\cot(\sqrt{k\,} r) \to - \infty$, as $r\to({\pi}/\sqrt{k\,}\,)^-$.
\end{proof}

\medskip

\noindent{\textbf{Riemannian manifolds with nonpositive sectional curvature}\smallskip\\
Another well--known classical result is the Cartan--Hadamard theorem, saying that the universal covering of a complete Riemannian manifold with nonpositive sectional curvature is diffeomorphic to $\R^n$. The key point of the proof is to show the following proposition.

\begin{proposizione}\label{CH}
Let $(M,g)$ be a complete Riemannian manifold with nonpositive sectional curvature. Then, every point $p\in M$ has no conjugate points.
\end{proposizione}
\begin{proof}[Proof via the distance functions]\ \\
Similarly to the alternative proof of Synge's lemma~\ref{Synge} above, we consider the minimal eigenvalue $\mu_{\min}$ of the matrix ${\mathrm S}_\alpha^\beta$ in polar coordinates centered at $p$, which is the minimal eigenvalue of the shape operator of the geodesic spheres around $p$.\\
Let $\gamma:[0,L]\to M$ be a unit--speed geodesic starting from $p$, up to its first conjugate point (if it exists).\\
If $\gamma$ is minimizing, the differential inequality~\eqref{eq333-0} holds,
$$
\partial_{r}\mu_{\min}(r)\geqslant -\mu_{\min}^2(r)
$$
for almost every $r\in (0,L)$. Therefore, by formula~\eqref{eq11111}, we have $\mu_{\min}(r)\geqslant1/r$ for every $r\in (0,L)$. Then, since (formula~\eqref{eq333-0_bis})
$$
\partial_{r}\lambda_{\min}(r)\geqslant2\lambda_{\min}(r)\mu_{\min}(r)
$$
for almost every $r$, where $\lambda_{\min}$ is the minimal eigenvalue of the matrix $g_{\alpha\beta}$, we conclude that $\lambda_{\min}(r)\geqslant\delta>0$ for every $r\in [L/2,L)$. It follows that there cannot be any point conjugate to $p$ along the geodesic $\gamma$, by the discussion at the end of Section~\ref{eriemsec}.\\
We refer the reader to \cite{Tesifra} for the general case where $\gamma$ is possibly non--minimizing.
\end{proof}

\medskip

\noindent{\textbf{Volume growth of geodesic balls}\smallskip\\
We deal with estimating the volume growth of the geodesic balls of complete Riemannian manifolds whose Ricci curvature is bounded below by a constant (which is not necessarily positive).

\begin{proposizione}[Bishop inequality]
If $(M,g)$ is a complete Riemannian manifold of dimension $n$, satisfying $\Ric \geqslant k(n-1)g$ for some $k\in\R$, then there holds 
$$
\vol(B_r(p))\leqslant\vol_k(B_r({\overline{p}}))\,,\label{stimaG}
$$
for every $p\in M$ and for all $r\in(0,{\pi}/{\sqrt{k\,}}\,)$. Here, $B_r(\overline{p})$ denotes the geodesic ball of radius $r>0$ centered at a point $\overline{p}\in\mathbb{M}_k^n$. 
\end {proposizione}
\begin{proof}[Proof via the distance functions]\ \\ 
We start by considering the polar coordinates $\varphi=(\theta,r)$ defined in the open set
$$
U=\left\{\exp_p(tv)\in M : \ v \in \SSSS^{n-1}_p\setminus\{N\} \text{ and } t\in(0,c(v)) \right\}\,,
$$
(recall that $c(v)$ is the first value such that the geodesic $\exp_p(tv)$ reaches the cut~locus of $p$), induced by the stereographic projection $(V,\theta)$ of $\SSSS^{n-1}_p$ from the north pole $N$.  Afterwards, we denote by $W$ the open subset $\varphi(U)\subseteq \R^n$ and we notice that 
$$
W\subseteq W_{k}=\left\{(\theta, t)\in \R^{n}\  : \ \theta \in \R^{n-1}\ \text{ and }\ \  t\in(0,{\pi}/{\sqrt{k\,}}\,)\, \right\}\,,
$$
by the Bonnet--Myers theorem, when $k>0$.\\ 
We set $G=\sqrt{\det g_{ij}}\circ\varphi^{-1}$, which is the density of the canonical volume measure of $(M,g)$ in the polar coordinates $\varphi$ on $U$. By construction, $G$ is then clearly defined on $W$.\\
Now, we recall that the Riemannian submanifold obtained from $\mathbb{M}_k^n$ by removing ${\overline{p}}$ and the cut~locus of ${\overline{p}}$ is isometric to $(M^{n}_{k},g^{n}_{k})$, given by formula \eqref{baseSF}. The latter is, in turn, isometric to
\begin{equation}
\big(\SSSS^{n-1}_p\times(0,\pi/\sqrt{k\,}\,),\,{\mathrm{sn}}_{k}^{2}(r)\,g_{\SSSS^{n-1}_p}+dr\otimes dr\big)\,.
\end{equation}
We then denote by $G_k(\theta,r)$ the analogue of $G(\theta,r)$ for the latter space, which is obtained by using the aforementioned chart $(V,\theta)$ of $\SSSS^{n-1}_p$ and is thus defined on $W_k$.\\
With the same notation as in the proof of Proposition~\ref{constsec}, along a unit--speed geodesic $r\mapsto\gamma(r)=\exp_{p}(rv)$ with $v=v(\theta)$, by formula~\eqref{eq:fff0}, there holds
\begin{equation}\label{eq:fff02}
\lim_{r \to 0^{+}}\frac{G^2(\theta, r)}{{\mathrm{sn}}^{2(n-1)}_{k}(r)}=\det g_{\alpha\beta}^{\SSSS^{n-1}_p}\!(v)\,,
\end{equation}
in particular,
$G(\theta,r)/G_k(\theta, r)\to 1$ as $r\to 0^+$.\\
Since, by Proposition~\ref{eqS}, we have
$$
\HHH(\theta, r)=\frac{\partial}{\partial r} \log G(\theta,r)\qquad\qquad\text{ and }\qquad\qquad \HHH_k(\theta, r)=\frac{\partial}{\partial r} \log G_k(\theta, r)\,,
$$
if $\Ric\geqslant(n-1) kg$, then by the mean curvature comparison inequality~\eqref{trSeqbis}, it follows that $\HHH(\theta, r) \leqslant \HHH_k(\theta, r)$, thus
\begin{equation}\label{eqcar7778}
\frac{\partial}{\partial r} \log G(\theta, r) \leqslant \frac{\partial}{\partial r} \log G_k(\theta, r)\,,\qquad\text{ that is, }\qquad\frac{\partial}{\partial r}\frac{G(\theta, r)}{G_k(\theta, r)}\leqslant 0\,.
\end{equation}
This clearly implies that $G(\theta,r )\leqslant G_k(\theta, r)$, and the conclusion follows by integration.
Indeed, these coordinates cover $M$ except for the cut~locus of $p$, the point $p$, and the radial geodesic in the direction $N$ up to its cut time. These sets have zero volume, since
$n\geqslant2$. Hence,
$$
\vol(B_r(p))=\int_{W_r} G(\theta,t)\,d\theta\,dt
$$
where $W_r=\varphi(U_r)\subseteq W$, with
$$
U_r=\Big\{\exp_p(tv)\in M : \ v \in \SSSS^{n-1}_p\setminus\{N\} \text{ and } t\in\big(0,\min \{r, c(v)\}\big) \Big\}\,.$$
Then, setting $\overline{t}(\theta)=\min\{r,c(v(\theta))\}$, the result obtained above, coupled with Fubini's theorem, implies
$$
\vol(B_r(p))\!=\!\!\int_{\R^{n-1}}\!\!\!\!\!\!\!d\theta\int_{0}^{\overline{t}(\theta)}\!\!\!G(\theta,t)\,dt\leqslant\!\int_{\R^{n-1}}\!\!\!\!\!\!\!d\theta\int_{0}^{\overline{t}(\theta)}\!\!\!G_k(\theta,t)\,dt\leqslant\!\int_{\R^{n-1}}\!\!\!\!\!\!\!d\theta\int_{0}^{r}\!\!G_k(\theta,t)\,dt\!=\!\vol_k(B_r({\overline{p}})).
$$
\end{proof}

\begin{osservazione}
Along the same lines, the {\em Bishop--Gromov inequality} also follows: the function
$$
r \mapsto \frac{\vol(B_r(p))}{\vol_k(B_r({\overline{p}}))}
$$
is nonincreasing, for $r>0$ (see~\cite[Theorem~III.4.5]{chavel}).
\end{osservazione}

\bibliographystyle{amsplain}
\bibliography{biblio}

\end{document}